\documentclass[12pt]{article}
\usepackage[cmtip,arrow]{xy}
\usepackage{pb-diagram, pb-xy}
\usepackage{amssymb,epsfig,amsfonts,amsmath}

\newtheorem{prop}{Proposition}
\newtheorem{nt}{Remark}
\newtheorem{Th}{Theorem}

\newfont{\ssdbl}{msbm8}
\newfont{\sdbl}{msbm9}
\newfont{\dbl}{msbm10 at 12pt}

\newcommand{\Ob}{\mathop {\rm Ob}}

\newcommand{\oo}{{\cal O}}

\newcommand{\Hom}{\mathop {\rm Hom}}
\newcommand{\Aut}{\mathop {\rm Aut}}

\newcommand{\End}{\mathop {\rm End}}

\newcommand{\Spec}{\mathop {\rm Spec}}

\newcommand{\Frac}{\mathop {\rm Frac}}

\newcommand{\da}{\mathbb{A}}

\newcommand{\Z}{\dz}

\newcommand{\lto}{\longrightarrow}

\newcommand{\df}{\mathbb{F}}
\newcommand{\zz}{\mathbf{Z}}

\newcommand{\D}{{\cal D}}
\newcommand{\M}{{\cal M}}

\newcommand{\p}{{\cal P}}

\newcommand{\el}{{\cal L}}

\def\Z{{\mathbb Z}}

\newcommand{\Gal}{{\rm Gal}}
\newcommand{\Fq}{{\mathbb F}_q}

\newcommand{\diag}{{\rm diag}}
\newcommand{\pic}{{\p ic}}
\newcommand{\picz}{{\p ic^\Z}}
\newcommand{\co}{{\rm Comm}}
\newcommand{\ddet}{{\D et}}
\newcommand{\Nm}{{\mathop{\rm Nm}}}

\newcommand{\quash}[1]{}

\begin{document}

\author{D. V. Osipov
\footnote{This work was partially  supported by
Russian Foundation for Basic Research (grants no.~11-01-00145-a,  no.~12-01-33024 mol\_a\_ved and no.~13-01-12420 ofi\_m2)
and by the Programme for the Support of Leading Scientific Schools of the Russian Federation (grant no.~NSh-5139.2012.1).}}

\title{Noncommutative reciprocity laws on algebraic surfaces: a case of tame ramification}
\date{}

\maketitle

\abstract{We prove non-commutative reciprocity laws on an algebraic surface defined over a perfect field. These reciprocity laws claim the splittings of some  central extensions of globally constructed groups over some subgroups constructed by points or projective curves on a surface. For a two-dimensional local field with a finite last residue field the constructed local central extension is isomorphic to a central extension which comes from the case of tame ramification of the Abelian two-dimensional local Langlands correspondence
suggested by M.~Kapranov.}

\section{Introduction}

The goal of this paper is to prove some non-commutative reciprocity laws on algebraic surfaces.
Let $X$ be a normal irreducible algebraic surface over a perfect field $k$. For any element $a $ from the multiplicative group of the field $k(X)$ of rational functions on $X$ we construct a central extension $\widehat{GL_{n,a}(\da_X)}$ of the group $GL_n(\da_X)$ by the group $k^*$.
Here $\da_X$ is the Parshin-Beilinson adelic ring of the surfaces $X$. By definition, we have
$$
\da_X \subset \prod_{x \in C} K_{x,C}  \mbox{,}
$$
where the product is taken of all pairs: a point $x \in X$ and an irreducible curve $C \subset X$ which contains the point $x$.
The ring $K_{x,C}$ is canonically  constructed by such a pair $\{ x \in C \}$ and this ring is a finite product of two-dimensional local fields.

Let $K_{x,C} = \prod\limits_{i=1}^l K_i$, where $K_i$ is a
two-dimensional local field. Let $K=K_i= k'((u))((t))$ for some $1
\le i \le l$. The restriction of the central extension
$\widehat{GL_{n,a}(\da_X)}$  (when $n \ge 2$) to the subgroup
$GL_n(K)$ of the group $GL_n(\da_X)$ is described by some element
from the group $\Hom(K_2(K), k^*)$. In proposition~\ref{prop} we
calculate this element, which  is given as the following map:
\begin{equation} \label{intr}
K_2(K)  \ni (f,g)  \mapsto \Nm_{k'/k} (f,g,a)_K  \in k^*  \mbox{,}
\end{equation}
where $(\cdot, \cdot, \cdot)_K$ is the two-dimensional tame symbol.

If $k'=\df_q$, then map~\eqref{intr}   coincides with the two-dimensional local reciprocity map for the Abelian Kummer extension $ K(a^{1/q-1})$ of the field $K$.

We prove noncommutative reciprocity laws in theorem~\ref{th1} for the central extension  $\widehat{GL_{n,a}(\da_X)}$. Let $x$ be  a point on $X$
and $K_x$ be a subring of the field $\Frac{\hat{\oo}_x}$ such that $K_x = k(X) \cdot \hat{\oo}_x$. The central extension $\widehat{GL_{n,a}(\da_X)}$
splits over the subgroup $GL_n(K_x)$ of the group $GL_n(\da_X)$. Let $C$ be a projective irreducible curve on $X$ and the field $K_C$ be the completion of the field $k(X)$ with respect to the discrete valuation given by the curve $C$.
The central extension $\widehat{GL_{n,a}(\da_X)}$
splits over the subgroup $GL_n(K_C)$ of the group $GL_n(\da_X)$. Besides, if $X$ is a projective surface, then the central extension $\widehat{GL_{n,a}(\da_X)}$
splits over the subgroup $GL_n(k(X))$ of the group $GL_n(\da_X)$.

We note that the main technical tool for this paper comes from
paper~\cite{OsZh}. In particular, we use categorical central
extensions of groups by the Picard groupoid of $\Z$-graded
one-dimensional $k$-vector spaces. These groups
 act on $2$-Tate $k$-vector spaces.

The paper is organized as follows. In \S~\ref{sec2} we recall some
explicit formulas from the two-dimensional local class field theory
for  Kummer extensions of two-dimensional local fields. In
\S~\ref{sec3} we recall some calculations of the group $H^2(GL_n(K),
k^*)$ (when $n \ge 2$) and relation of  central extensions of this
group with the algebraic $K$-theory of the field $K$. In
\S~\ref{sec4} we recall some categorical notions and constructions
from~\cite{OsZh}. In \S~\ref{centr} we construct the central
extension $\widehat{GL_{n,a}(\da_X)}$ and study its main properties
(see propositions~\ref{restr}  and~\ref{prop}). We prove also the
noncommutative reciprocity laws (see theorem~\ref{th1}).

\section{Some explicit formulas of two-dimensional class field theory} \label{sec2}
Let $K= \Fq((u))((t))$ be a two-dimensional local field, where $\Fq$ is a finite field.
We suppose that $\mu_m \subset \Fq^*$, where the group $\mu_m$ is the group of all roots of unity of degree $m$.
We consider an element $a \in K^*$ and an extension $L = K(a^{1/m})$. By the Kummer theory, the field $L$ is
a Galois extension of the field $K$ and $$\Gal (L/K) \cong \mu_l  \mbox{,}$$ where $\mu_l \subset \mu_m$ and $l \mid m$.

By the local class field theory for the two-dimensional local  field $K$ (see~\cite{Pa}) there exists a
surjective reciprocity map
$$
\phi_{L/K} \; : \; K_2(K)  \longrightarrow \Gal (L/K)  \mbox{.}
$$
The map $\phi_{L/K}$ is explicitly written as
\begin{equation}  \label{symb}
\phi_{L/K} ((f,g))  = {(f,g,a)_K}^{\frac{q-1}{m}}  \mbox{,}
\end{equation}
where elements $f$ and $g$ are from $K^*$ and
$$
(\cdot, \cdot, \cdot)_K  \; : \; K_3^M(K) \lto \Fq^*
$$
is the two-dimensional tame symbol. The map $(\cdot, \cdot, \cdot)_K$ is the composition of the boundary maps in
Milnor K-theory and it is defined for any field $K=k((u))((t))$ where we changed the field $\Fq$ to a field $k$ and  values of  $(\cdot, \cdot, \cdot)_K$ are in $k^*$. An explicit formula for the map  $(\cdot, \cdot, \cdot)_K$ is the following:
\begin{equation} \label{2tame}
(f,g,h)_K = {\rm sgn}(f,g,h) f^{\nu_K(g,h)} g^{\nu_K(h,f)} h^{\nu_K(f,g)} \mod m_K  \mod m_{\bar{K}}  \mbox{,}
\end{equation}
where
elements $f$, $g$ and $h$ are from $K^*$,
$m_K$ is the maximal ideal of the discrete valuation ring of the field $K$, $m_{\overline{K}}$ is the maximal ideal of the discrete valuation ring of the field $\overline{K}=k((u))$. Besides, ${\rm{sgn}} (f,g,h) = (-1)^A$, where
$$
A = \nu_K(f,g)\nu_K(f,h) + \nu_K(g,h) \nu_K(g,f) + \nu_K(h,f) \nu_K(h,g) +
\nu_K(f,g) \nu_K(g,h) \nu_K(h,f)  \mbox{,}
$$
 and the map $\nu_K(\cdot, \cdot) : K_2(K)  \lto \Z$ is also a composition of some boundary maps in Milnor K-theory
 with an explicit formula:
 $$
 \nu_K(f,g) = \nu_{\overline{K}} \left(  \frac{f^{\nu_K(g)}}{g^{\nu_K(f)}} \mod m_K \right)  \mbox{,}
 $$
where $\nu_K : K^* \to \Z$ and $\nu_{\overline{K}} : \overline{K}^* \to \Z$ are discrete valuations of the corresponding fields.

\section{Reminding on some calculations}  \label{sec3}
We change the ground field $\df_q$
to a perfect field $k$.
We recall some calculations from~\cite[\S~2.2]{Os}.
For any $n \ge 2$ we have the following formula from remark~2 of~\cite{Os}:
\begin{equation}  \label{coh}
H^2(GL_n(K), k^*) = H^2(K^*, k^*)  \oplus \Hom(K_2(K), k^*)  \mbox{.}
\end{equation}

We will be interested in central extensions of the group $GL_n(K)$ by the group $k^*$ such that these central extensions (up to isomorphism) come from elements of the group  $\Hom(K_2(K), k^*)   $  via formula~\eqref{coh}.
Such central extensions are in one-to-one correspondence with  central extensions which are isomorphic to the trivial cental extension after the restriction to the subgroup $K^*$ of the group $GL_n(K)$, where the group $K^*$ is embedded in the upper left corner of  the matrix group $GL_n(K)$.

\begin{nt} \label{rem}
{\em We consider in this paper the subgroup $K^*$ embedded into the upper left corner of the group $GL_n(K)$.
Since an inner automorphism of the group $GL_n(K)$ does not change the isomorphism class of any central extension
of this group, we could fix an embedding of the group $K^*$ into any other place of the diagonal. We will obtain the same results.}
\end{nt}

If a central extension
$$
1 \lto k^* \lto \hat{G}  \lto GL_n(K)  \lto 1
$$
splits  (i.e. it is isomorphic to trivial) over the subgroup $K^* \hookrightarrow GL_n(K)$, then the corresponding element
from $\Hom(K_2(K), k^*)$, which we call the symbol, is obtained as following:
\begin{equation}  \label{form}
\{x,y\} = <\diag (y,1, \ldots, 1), \diag(1,x,1, \ldots, 1)>   \mbox{,}
\end{equation}
where $x$ and $y$ are from $K^*$, $\{x,y \} \in k^*$
and $<\cdot, \cdot>$ is the commutator of the lifting of two commuting elements from $GL_n(K)$ to $\hat{G}$.

\medskip

Our goal in this paper is to construct a central extension of the group $GL_n(K)$ without using of algebraic $K$-theory such that this central extension will correspond to the symbol given by  map~\eqref{symb} (when $m =q-1$, then map~\eqref{symb} makes sense for the field $k$, because this map comes from the two-dimensional tame symbol,  and  map~\eqref{symb}  depends on an element $a \in K^*$).  This construction has also to be true if we change the field $K$ to the ring $\da_X$  of the Parshin-Beilinson adeles of an algebraic surface $X$ defined over the field $k$. We will also prove
the reciprocity laws for central extensions which will be constructed in such a way.

\begin{nt} {\em
When $n=2$ and $k=\Fq$, starting from a symbol given by map~\eqref{symb}  we obtain a central extension which correspond to the Abelian two-dimensional local
Langlands correspondence suggested by M.~Kapranov, see~\cite{K} and also~\cite{Os}.}
\end{nt}

\section{Reminding on some categorical notions} \label{sec4}

We recall some notions from~\cite{OsZh}.

By definition, a Picard groupoid $\p$ is a symmetric monoidal
group-like groupoid. We will consider the following two examples of Picard groupoids.
\begin{itemize}
\item
$\p = \pic $ is the groupoid of one-dimensional $k$-vector spaces where $k$ is a field. (Here groupoid means that we consider only isomorphisms in the category of one-dimensional $k$-vector spaces.)

\item
$\p = \picz$ is the groupoid of $\Z$-graded one-dimensional $k$-vector spaces. In other words, $X \in \Ob(\picz)$
iff $X= (l,n)$ where $l \in \Ob(\pic)$, $n \in \Z$ and
$$ \Hom\nolimits_{\picz} ((l_1, n_1), (l_2,n_2)) = \left\{
\begin{array}{rcl}
\Hom\nolimits_{\pic}(l_1,l_2)= \Hom_k(l_1,l_2) \setminus 0  & \mbox{if} & n_1=n_2 \\
\emptyset  &  \mbox{if} & n_1 \ne n_2
\end{array}
 \right. $$
We put $(l_1,n_1) \otimes (l_2, n_2) = (l_1 \otimes l_2, n_1 + n_2)$.
\end{itemize}

The main difference between Picard groupoids $\pic$ and $\picz$ is  commutativity constraints:
$X \otimes Y \to Y \otimes X$ (where $X$ and $Y$ are either objects of $\pic$ or objects of $\picz$).
Indeed,
\begin{itemize}
\item if $X=l_1$, $Y =l_2$ are objects from  $\pic$, then the commutativity constraint   $c_{l_1,l_2}: l_1 \otimes l_2 \to l_2 \otimes l_1$ is the usual isomorphism; \\
\item if $X=(l_1, n_1)$, $Y = (l_2, n_2)$ are objects from $\picz$, then the commutativity constraint
is equal to
 $$(-1)^{n_1n_2} c_{l_1, l_2}  \in
 \Hom\nolimits_{\picz}((l_1 \otimes l_2, n_1 + n_2), (l_2 \otimes l_1, n_1 + n_2)) \mbox{.}$$

\end{itemize}

We introduce a natural Picard groupoid $\zz$. The objects of $\zz$ are elements of the group $\Z$, i.e. integer numbers.
The morphisms in $\zz$ are only the identity morphisms. In other words, for any integers $m$ and $n$ we have $$\Hom\nolimits_{\zz} (m,n)= \left\{
\begin{array}{rcl}
{\rm id}  & \mbox{if} & m=n \\
\emptyset &  \mbox{if}         & m \ne n \mbox{.}
\end{array}
    \right.  $$
We note that there is a  symmetric monoidal functor between Picard groupoids:
$$
\psi  \; : \; \picz \lto \zz      \quad   \mbox{,} \quad \psi((l,n))=n   \mbox{.}
$$

Besides, there is a monoidal functor
$$F_{\pic} \; :\; \picz  \lto \pic  \quad   \mbox{,} \quad F_{\pic}((l,n))=l  \mbox{.} $$
The functor $F_{\pic}$ is not symmetric, i.e. this functor does not preserve the commutativity constraints.

\medskip

Let $G$ be a group and $\p$ be a Picard groupoid.
We recall (see~\cite[\S~2D]{OsZh}) that a Picard groupoid $H^1(BG, \p)$ is the groupoid of  monoidal functors
from the discrete monoidal category $G$ to the monoidal category $\p$. In other words, let $f$ be an object of
  $H^1(BG, \p)$, then for any element $g \in G$ we have $f(g) \in \Ob(\p)$ and for any $g_1, g_2 \in G$
  we have an isomorphism $f(g_1g_2) \simeq f(g_1) + f(g_2)$\footnote{Here and in the sequel we use also the notation $+$ for the monoidal product in a Picard groupoid.} compatible with the associativity condition in $\p$.

  If $\p = \pic$, then the groupoid $H^1(BG, \p)$
   is equivalent to the groupoid of central extensions of the group $G$ by the group $k^*$. Hence,
  the group  of isomorphism classes of objects of $H^1(BG, \p)$ is isomorphic to the group  $H^2(G, k^*)$.
   We note that the functor $F_{\pic}$ induces a functor (which is not monoidal):
   \begin{equation} \label{red}
    H^1(BG, \picz)  \lto H^1(BG, \pic) \quad :  \quad f \mapsto F_{\pic} \circ f \mbox{.}
   \end{equation}

Let $f$ be an element  from the set $\Ob(H^1(BG, \p)) $ (where $\p$ is any Picard groupoid). For any $g_1, g_2 \in G$ such that $[g_1, g_2]=1$
there is an element $\co(f)(g_1, g_2) \in \Aut_{\p}(e) = \End_{\p}(e)$,
where $e$ is a unit object in $\p$  (see~\cite[Lem.-Def.~2.5]{OsZh}). The map $\co(f)$ is
 a map from the set of pairs of commuting elements of $G$ to the Abelian  group $\Aut_{\p}(e)$. This map
 is
 a bimultiplicative and antisymmetric map with respect to elements $g_1$ and $g_2$.

If $\p = \pic$, then $\co(f)(g_1,g_2)=<g_1, g_2> \in k^*$, where $<g_1, g_2>$ is the commutator of the lifting of elements $g_1$ and $g_2$ to the central extension of the group $G$ by the group $k^*$ such that this central extension corresponds to $f$.

If $\p = \picz$, then $\co(f)(g_1,g_2) \in k^*$. Moreover,
from the construction of the map $\co(f)$  we have the following property.
If $f(g_1)= (l_1, n_1) \in \Ob(\picz)$,
$f(g_2)= (l_2, n_2) \in \Ob(\picz)$ where $n_i \in \Z$, then after the composition of $f$ with the functor $F_{\pic}$
we obtain
\begin{equation}  \label{dvad}
\co(f)(g_1, g_2)= (-1)^{n_1n_2} \co( F_{\pic} \circ f)(g_1, g_2)= (-1)^{n_1n_2}<g_1, g_2>  \mbox{,}
\end{equation}
 where $<g_1, g_2>$ is the commutator of the lifting of elements $g_1$ and $g_2$ to the central extension of the group $G$ by the group $k^*$ such that this central extension corresponds to $\picz \circ f$.

\medskip

Let $G$ be any group and $\p$ be any Picard groupoid. We recall (see~\cite[\S~2E]{OsZh})
that objects of the Picard groupoid $H^2(BG, \p)$ are (categorical) central extensions $\el$ of the group $G$ by the Picard groupoid
$\p$:
$$
1 \lto \p \stackrel{i}{\lto} \el \stackrel{\pi}{\lto} G \lto 1  \mbox{,}
$$
where $\el$ is a group-like monoidal groupoid, $G$ is a discrete group-like monoidal groupoid constructed by the group $G$, $i$ and $\pi$ are monoidal functors. More explicitly, a central extension $\el$  can be given in the following way. For any $g \in G$ we have a $\p$-torsor $\el_g= \pi^{-1}(g)$ and for any $g_1, g_2 \in G$  we have a natural equivalence of $\p$-torsors:
$$
\el_{g_1 g_2}  \simeq \el_{g_1} + \el_{g_2}
$$
together with  some further isomorphisms between equivalences and   compatibility conditions  on these  isomorphisms. We have $\Ob(\el)= \bigcup_{g \in G} \Ob(\el_g)$. If $X \in \Ob(\el_{g_1})$ and $Y \in \Ob(\el_{g_2})$, then
$$
\Hom\nolimits_{\el}(X,Y) = \left\{
\begin{array}{rcl}
\Hom_{\el_{g_1}}(X,Y)  & \mbox{if} & g_1=g_2 \\
\emptyset  & \mbox{if} & g_1 \ne g_2  \mbox{.}
\end{array}
\right.
$$

\begin{nt}  \em
If $\p = \pic $, then the group of isomorphism classes of objects of $H^2(BG, \p)$ is isomorphic
to the group $H^3(G, k^*)$  (see, for example,~\cite{Br2}).
\end{nt}

We note that the functor $\psi: \picz \lto \zz$ induces a symmetric monoidal functor $\Psi$ from the Picard $2$-groupoid of $\picz$-torsors to the Picard groupoid of $\zz$-torsors by the rule:
$$
\Psi(\M) = \zz
\stackrel{\picz}{\wedge}
\M   \mbox{,}
$$
where $\M$ is a $\picz$-torsor (see details in~\cite[\S~6.7]{Br1}). The Picard groupoid of $\zz$-torsors is equivalent to the Picard groupoid of $\Z$-torsors. Using the functor $\Psi$ we obtain a symmetric monoidal functor
$\tilde{\Psi}$ from the Picard groupoid $H^2(BG, \picz)$ to the Picard groupoid of central extensions of the group $G$
by the group $\Z$:
 $$
 \el \in H^2(BG, \picz) \mbox{,} \qquad \tilde{\Psi}(\el)(g)= \Psi(\el_g)   \quad \mbox{for any}  \quad g \in G
 $$
 (We note that  the group of isomorphism classes of objects of the  Picard groupoid of central extensions of the group $G$
by the group $\Z$ is isomorphic to $H^2(G, \Z)$.)

 For any object $\el$ of the Picard groupoid  $H^2(BG, \p)$,
for any $g_1, g_2 \in G$ such that $[g_1,g_2]=1$ there is an object $C_2^{\el}(g_1, g_2)$ of the Picard groupoid $\p$ (see~\cite[Lem.-Def.~2.13]{OsZh}). The morphism $C_2^{\el}$ is a bimultiplicative and antisymmetric (with respect to $g_1$ and $g_2$)
morphism from the set of pairs of commuting elements of $G$ to the Picard groupoid $\p$. An object
$C_2^{\el}(g_1, g_2)$ comes from
an auto-equivalence of the $\p$-torsor $\el_{g_1 g_2}$. This auto-equivalence is
the composition of the following equivalences of $\p$-torsors:
\begin{equation}  \label{for1}
\el_{g_1 g_2}  \simeq \el_{g_1} + \el_{g_2}  \simeq \el_{g_2} + \el_{g_1}  \simeq \el_{g_2g_1}= \el_{g_1g_2}  \mbox{.}
\end{equation}

\begin{nt} \em
If $\p = \pic$, then the morphism  $C_2^{\el}$
is a weak biextension of the set of pairs of commuting elements of the group $G$ by the group $k^*$.
This weak biextension is described by
  partial symmetrizations of the $3$-cocycle
corresponding to $\el$ (see~\cite{Br2}).
\end{nt}

If $\p=\picz$ and $\el \in \Ob(H^2(BG, \p))$, then it is easy to see from formula~\eqref{for1} that after application of the functor $\psi: \picz \lto \zz$ we obtain
\begin{equation}  \label{trid}
\psi(C_2^{\el}(g_1,g_2)) = -<g_1, g_2>  \; \in  \; \Z  \mbox{,}
 \end{equation}
 where $<g_1, g_2>$ is the commutator of lifting of elements $g_1$ and $g_2$ to the central extension $\tilde{\Psi}(\el)$
 of the group $G$ by the group $\Z$.

 Let $G$ be any group, $\p$ be any Picard groupoid and a central extension $\el$ be from $\Ob(H^2(BG, \p))$. We fix an element $ g \in G$.  We denote by
 $Z_G(g) \subset G$ a subgroup which is the centralizer of  the element $g$. From morphism $C_2^{\el}$ we obtain an object
 $C_g^{\el}$ of    $H^1(BZ_G(g), \p) $ by the rule:
 \begin{equation}  \label{biext}
  \mbox{for any}  \quad h \in Z_G(g) \quad \mbox{we put} \quad C_g^{\el}(h)= C_2^{\el}(g,h) \in \Ob(\p) \mbox{.}
  \end{equation}

 For any $g_1, g_2, g_3 \in G$ such that $[g_i, g_j]=1$ ($1 \le i,j \le 3$) there is an element
 \begin{equation}  \label{C3}
 C_3^{\el}(g_1,g_2,g_3)= \co(C_{g_1}^{\el})(g_2,g_3) \in \Aut\nolimits_{\p}(e) \mbox{.}
 \end{equation}
The map $C_3^{\el}$ is
 a map from the set of triples of  pairwise commuting elements of $G$ to the Abelian  group $\Aut_{\p}(e)$. This map
 is
 a trimultiplicative and antisymmetric map with respect to elements $g_1$, $g_2$ and $g_3$
 (see~\cite[Prop.~2.17]{OsZh}).

\section{Central extensions} \label{centr}
Let $X$ be a normal irreducible algebraic surface  over a perfect field $k$ (in particular, $X$ is an integral scheme). We will try to keep the notation similar to~\cite[\S~3.2]{Os}.

Let $\Delta$ be some  subset of all pairs $\{ x \in C \}$ where $x \in X$ is a point and $C \subset X$ is an irreducible curve which contains the point $x$ (in other words, $C$ is an integral one-dimensional subscheme of $X$). Let $K_{x,C}$ be the ring canonically associated
with a pair $\{x \in C \}$ from  $\Delta$ (see, for example,~\cite[\S~2.2]{O4}). The ring $K_{x,C}$ is a finite product of two-dimensional local fields such that every two-dimensional local field from this product corresponds to a branch of the curve $C$ restricted to the formal neighbourhood of the point $x$. In particular, if $x$ is a regular point on $C$  and on $X$, then $K_{x,C}$
is a two-dimensional local field isomorphic to $k(x)((u))((t))$ where $k(x)$ is the reside field of the point $x$,
$t =0$ is a local equation of the curve $C$ in some affine neighbourhood of the point $x$ on $X$, $u=0$ is a local equation of a curve which is defined in  the same neighbourhood and is  transversal  to $C$ in $x$.
If $K_{x,C} = \prod\limits_i K_i$ where $K_i$ is a two-dimensional local field, then we denote
$\oo_{K_{x,C}}= \prod\limits_i \oo_{K_i}$ where $\oo_{K_i}$ is the discrete valuation ring of the field $K_i$.
In particular, if $x$ is a regular point on $C$  and on $X$, then the ring $\oo_{K_{x,C}}$ is isomorphic to  the ring $k(x)((u))[[t]]$.

Let $\da_X$  be the Parshin-Beilinson adelic ring of $X$ (see, for example, survey~\cite{O4}).
We have an embedding
$$
\da_X \subset \prod_{\{x \in C \}}  K_{x,C}  \mbox{.}
$$
Inside of the ring $\prod_{\{x \in C \}}  K_{x,C}$ we define subrings  $\da_{\Delta}$ and $\oo_{\da_{\Delta}}$ (which depend on the set  $\Delta$)\footnote{From~\cite[prop.~3.4, prop.~3.8]{O4} we have that the ring $\oo_{\da_{\Delta}}$ coincides with the adelic ring
defined by the construction of A.~A.~Beilinson (see~\cite[\S~3.2]{O4}) by means of the set of $2$-simplices $(C,x)$ (where $\{ x \in C \} \in \Delta$) of the simplicial set which is defined by scheme points of the surface  $X$.
Analogously, the ring  $\da_{\Delta}$ coincides with the adelic ring defined by the set of  $3$-simplices $(X,C,x)$ (where $\{ x \in C \} \in \Delta$). In our notation we identify irreducible varieties and corresponding generic scheme points.}:
$$
\da_{\Delta} = \da_X \cap \prod_{\{x \in C \} \in \Delta}  K_{x,C}  \qquad \mbox{and} \qquad
\oo_{\da_{\Delta}} = \da_X \cap \prod_{\{ x \in C \} \in \Delta}  \oo_{K_{x,C}}  \mbox{.}
$$

If $\Delta_1$ is a subset of the set $\Delta$, then we have
$$
\da_{\Delta} = \da_{\Delta_1}  \times \da_{\Delta \setminus \Delta_1}
\qquad \mbox{and}  \qquad
\oo_{\da_{\Delta}} = \oo_{\da_{\Delta_1}}  \times \oo_{\da_{\Delta \setminus \Delta_1}}  \mbox{.}
$$
Hence for any $n \ge 1$ we obtain
\begin{equation} \label{dec}
GL_n(\da_{\Delta}) = GL_n(\da_{\Delta_1})  \times GL_n(\da_{\Delta \setminus \Delta_1}) \mbox{.}
\end{equation}

Let $i_{\Delta_1, \Delta,n} : GL_n(\da_{\Delta_1})  \hookrightarrow  GL_n(\da_{\Delta})$
and $j_{\Delta_1, \Delta, n} : GL_n(\da_{\Delta})  \twoheadrightarrow GL_n(\da_{\Delta_1})$
be the embedding and the surjection of the groups with respect to decomposition~\eqref{dec}.

We have that
$\da_{\Delta}^n$ is a $C_2$-space over the field $k$ (see~\cite[Th.~1]{Osip}).
Besides, the group $GL_n (\da_{\Delta})$ acts on the space $\da_{\Delta}^n$ and this action is given by automorphisms of the $C_2$-space.
Moreover, $\da_{\Delta}^n$
is a complete $C_2$-vector space, i.e. it is a $2$-Tate vector space over the field $k$. The subspace
$\oo_{\da_{\Delta}}^n  \subset  \da_{\Delta}^n$  is a lattice in $\da_{\Delta}^n$. We fix this lattice. Now as in~\cite[\S~4C]{OsZh} we canonically construct (by means of Kapranov's graded determinantal theories) a categorical central extension $\ddet_{\da_{\Delta}^n}$ of the group
$GL_n(\da_{\Delta})$ by  the Picard groupoid $\picz$, i.e. we construct
$$
\ddet_{\da_{\Delta}^n} \; \in \; \Ob(H^2(BGL_n(\da_{\Delta}), \picz))  \mbox{.} $$

We fix some element $a\in \da_{\Delta}^*$. The element $\diag(a, \ldots, a)$
belongs to the center of the group $GL_n(\da_{\Delta})$. Therefore
$Z_{GL_n(\da_{\Delta})}(\diag(a, \ldots, a)) = GL_n(\da_{\Delta})$. Using formula~\eqref{biext} we obtain the following  object from $H^1(BGL_n(\da_{\Delta}), \picz)$:
\begin{equation} \label{ff3}
\ddet_{a,n} = C_{\diag(a, \ldots, a)}^{\ddet_{\da_{\Delta}^n}}  \mbox{.}
\end{equation}

We apply the functor $F_{\picz} : \picz \to \pic$ and obtain as in formula~\eqref{red}:
$$
\widetilde{\ddet_{a,n}} = F_{\picz}  \circ \ddet_{a,n}  \; \in \; \Ob(H^1(BGL_n(\da_{\Delta}), \pic))  \mbox{.}
$$
The object $\widetilde{\ddet_{a,n}}$ is given  by the central extension:
\begin{equation}  \label{tilcexp}
1 \lto k^*  \lto \widetilde{GL_{n,a}(\da_{\Delta})}   \stackrel{\theta}{\lto}
GL_n(\da_{\Delta})  \lto 1 \mbox{.}
\end{equation}
The   central extension~\eqref{tilcexp}  defines an element from the group $H^2(GL_n(\da_{\Delta}), k^*)$.

We know that $GL_n(\da_{\Delta}) = SL_n(\da_{\Delta}) \rtimes \da_{\Delta}^*$,
where the group $\da_{\Delta}^*$ is embedded into the upper left corner of the matrix group
$GL_n(\da_{\Delta})$ and acts by conjugations. We define a group
\begin{equation} \label{ff4}
\widehat{ GL_{n,a}(\da_{\Delta}) }= \theta^{-1}(SL_n(\da_{\Delta})) \rtimes \da_{\Delta}^* \mbox{.}
\end{equation}
The group $\widehat{ GL_{n,a}(\da_{\Delta}) }$ is a central extension of the group $GL_n(\da_{\Delta})$ by  the group $k^*$. This central extension splits over the subgroup $\da_{\Delta}^*  \subset GL_n(\da_{\Delta})$. We note that similar to remark~\ref{rem} we have that the embedding of the group $\da_{\Delta}^*$ into the other place of the diagonal of the matrix group $GL_n(\da_{\Delta})$ does not change the isomorphism class of the central extension $\widehat{ GL_{n,a}(\da_{\Delta}) }$ (compare also with remark~3 of~\cite{Os}).

\begin{prop} \label{restr}
We fix some element $a \in \da_{\Delta}^*$, where $\Delta$ is a subset of the set of all pairs
$\{ x\in C \}$ on $X$ as in the beginning of this section.
\begin{enumerate}
\item \label{delta}
Let $\Delta_1 \subset \Delta$ be a subset. The restriction of the central extension
$\widehat{ GL_{n,a}(\da_{\Delta}) }$ on the subgroup $GL_n(\da_{\Delta_1})$
embedded into the group $GL_{n,a}(\da_{\Delta})$ by
the homomorphism $i_{\Delta_1, \Delta, n}$ is isomorphic to the central extension
$\widehat{ GL_{n,j_{\Delta_1, \Delta,n}(a)}(\da_{\Delta_1}) }$
of the group $GL_n(\da_{\Delta_1})$ by the group $k^*$.
\item
Let $1 \le m \le n$. The restriction of the central extension
$\widehat{ GL_{n,a}(\da_{\Delta}) }$ on the subgroup $GL_m(\da_{\Delta})$
of the group $GL_n(\da_{\Delta})$ (embedded into the upper left corner)
is isomorphic to the central extension  $\widehat{ GL_{m,a}(\da_{\Delta}) }$
of the group $GL_m(\da_{\Delta})$ by the group $k^*$.
\end{enumerate}
\end{prop}
\begin{proof}
Clearly, it is enough to prove the proposition for the central extensions $\widetilde{GL_{n,a}(\da_{\Delta})}$. Since the functor given by formula~\eqref{red}
commutes with the restriction   on a subgroup $H \subset G$, it is enough to proof the analogous
statements for $\ddet_{a,n}  \in \Ob(H^1(BGL_n(\da_{\Delta}), \picz))$. In the last case the statements of the proposition follow from Lemma~2.15 of~\cite{OsZh} which claims that
the morphism $C_2^{\el_1 + \el_2}$ is isomorphic to the morphism $C_2^{\el_1} + C_2^{\el_2}$ for any $\el_1 , \el_2 \in \Ob(H^2(BG, \picz))$ and from Lemma~4.9 of~\cite{OsZh}
which claims that if a group $G$ acts on $2$-Tate vector spaces $V_1$ and $V_2$,
 then
 \begin{equation}  \label{detfor}
 \ddet_{V_1} + \ddet_{V_2} \simeq
 \ddet_{V_1 \oplus V_2}
 \end{equation}
 in $H^2(BG, \picz)$.
 The proposition is proved.
\end{proof}

\bigskip

We consider now a case when the set $\Delta$ is a singleton, i.e. it corresponds to a pair $\{x \in C\}$ on $X$. We have $\da_{\Delta}=K_{x,C}= \prod\limits_{i=1}^l K_i$, where $K_i$ is a two-dimensional local filed. Let $K =K_i=k'((u))((t))$ for some
$i$, where $k' \supset k$ is a finite extension. Let the central extension $\widehat{ GL_{n,a}(K) }$ of the group $GL_n(K)$ by the group $k^*$ be a restriction of the central extension $\widehat{ GL_{n,a}(\da_{\Delta}) }$
 of the group $GL_{n}(\da_{\Delta})$ by the group $k^*$
 with respect to the morphism $GL_n(K) \hookrightarrow   GL_{n}(\da_{\Delta}) $. We note that similar to statement~\ref{delta} of proposition~\ref{restr} the central extension $\widehat{ GL_{n,a}(K) }$ depends on the image of the element $a \in K_{x,C}^*$ in $K^*$. Therefore we will assume that $a \in K^*$ in notation of the central extension  $\widehat{ GL_{n,a}(K) }$.

 \begin{prop} \label{prop}
 We fix an integer $n \ge 2$ and an element $a \in K^*$. Then  the central extension
 $\widehat{ GL_{n,a}(K) }$ of the group $GL_n(K)$ by the group $k^*$
 corresponds to the following element from $\Hom (K_2(K), k^*)$ (with respect to decomposition~\eqref{coh}):
 $$
 K_2(K) \ni (f,g) \; \mapsto \; \Nm_{k'/k} (f,g,a)_K  \in k^*  \mbox{,}
 $$
 where $(\cdot, \cdot, \cdot)_K$ is the two-dimensional tame symbol (see formula~\eqref{2tame}).
 \end{prop}
\begin{proof}
Since, by  construction, the central extension
 $\widehat{ GL_{n,a}(K) }$  splits over the subgroup $K^*
 \subset GL_{n}(K)$, we have to calculate an expression given by formula~\eqref{form}.
 We have
\begin{multline*}
\{f,g  \} = < \diag(g,1, \ldots, 1), \diag(1,f,1, \ldots, 1)>= \\ =
<\diag(g,1, \ldots,1),\diag(f,1, \ldots,1)>
<\diag(g,1, \ldots, 1), \diag(f^{-1},f,1, \ldots,1)> \mbox{.}
\end{multline*}
Since the central extension
 $\widehat{ GL_{n,a}(K) }$  splits over the subgroup $K^*
 \subset GL_{n}(K)$, we have $<\diag(g,1, \ldots,1),\diag(f,1, \ldots,1)>=1$. Therefore we  have to calculate
 the value of
 $<\diag(g,1, \ldots, 1), \diag(f^{-1},f,1, \ldots,1)>$. Since the group
 $\widehat{ GL_{n,a}(K) }$ is a semidirect product,  the value of the last expression coincides with the value of the same expression  but calculated in the central extension $\widetilde{ GL_{n,a}(K) }$ (where the central extension $\widetilde{ GL_{n,a}(K) }$  of the group $GL_n(K)
 $ by the group $k^*$ is the restriction of the central extension
 $\widetilde{ GL_{n,a}(K_{x,C}) }$ on the subgroup $GL_n(K) \hookrightarrow GL_n(K_{x,c})$).
 Hence and from formula~\eqref{dvad}    we have
 \begin{multline*}
 <\diag(g,1, \ldots, 1), \diag(f^{-1},f,1, \ldots,1)>=  \\ = \co
 (\widetilde{\ddet_{a,n}})
 (\diag(g,1, \ldots, 1), \diag(f^{-1},f,1, \ldots,1))= \\
 = (-1)^B \co
 ({\ddet_{a,n}})
 (\diag(g,1, \ldots, 1), \diag(f^{-1},f,1, \ldots,1))  \mbox{,}
 \end{multline*}
 where
 $$
 B= \psi ( \ddet_{a,n} (\diag(g,1, \ldots, 1)))  \cdot
 \psi ( \ddet_{a,n} (\diag(f^{-1},f,1, \ldots, 1)))  \mbox{.}
 $$
  From formulas~\eqref{ff3}, \eqref{biext} and \eqref{trid}   we have
 \begin{multline*}
 \psi ( \ddet_{a,n} (\diag(f^{-1},f,1, \ldots, 1)))=
 \psi \left(C_{\diag(a, \ldots, a)}^{\ddet_{\da_{\Delta}^n}} (\diag(f^{-1},f,1, \ldots, 1))
 \right) = \\ =
 \psi \left(C_2^{\ddet_{\da_{\Delta}^n}} ({\diag(a, \ldots, a)} ,\diag(f^{-1},f,1, \ldots, 1)) \right)= \\ =
 -< \diag(a, \ldots, a),  \diag(f^{-1}, f. \ldots, 1)   >_{\Z}   \mbox{,}
\end{multline*}
where $< \cdot ,  \cdot >_{\Z}$ is the commutator of the lifting of commuting elements from the group $GL_n(K)$ to the central extension $\tilde{\Psi}(\ddet_{\da_{\Delta}^n})$
of the group $GL_n(K)$ by the group $\Z$. From formula~\eqref{detfor}
 we have that
 \begin{multline*}
 < \diag(a, \ldots, a),  \diag(f^{-1}, f, \ldots, 1)   >_{\Z}= \\=
 < \diag(a,1 , \ldots, 1), \diag(f^{-1}, 1, \ldots, 1)  >_{\Z}   +\\+ <\diag(1, a, 1, \ldots, 1),
 \diag(1,f,1, \ldots, 1)>_{\Z} =0 \mbox{.}
 \end{multline*}
Therefore $B=0$. Now, using formulas~\eqref{C3} and~\eqref{detfor} we have
\begin{multline*}
\co
 ({\ddet_{a,n}})
 (\diag(g,1, \ldots, 1), \diag(f^{-1},f,1, \ldots,1)) = \\ = C_3^{\ddet_{\da_{\Delta}^n} }
 (\diag(a, \ldots, a), \diag(g,1, \ldots, 1), \diag(f^{-1},f,1, \ldots,1))= \\ =
 C_3^{\ddet_K}(a,g,f^{-1}) =  C_3^{\ddet_K} (f,g,a) =\Nm_{k'/k} (f,g,a)_K \mbox{,}
\end{multline*}
where the last equality follows from~\cite[Th.~4.11]{OsZh}. From the above calculations we obtain
 $\{f,g  \} =\Nm_{k'/k} (f,g,a)_K$. The proposition is proved.
\end{proof}

\bigskip
Let $X$ be an algebraic surface as in the beginning of \S~\ref{centr}. By $k(X)$
we denote the field of rational functions on $X$. For any irreducible curve $C$ on $X$
we denote by $K_C$ a field which is the completion of the field $k(X)$ with respect to  a discrete valuation given by the curve $C$. For any point $x \in X$ we denote by $\hat{\oo}_x$ the completion of the local ring $\oo_x$ of the point $x$ on $X$ with respect to the maximal ideal. We introduce a ring
$K_x= k(X) \cdot \hat{\oo}_x$ where the product is taken in the field $\Frac(\hat{\oo}_x)$.
 We fix an integer $n \ge 1$. There is the diagonal embedding of the group
$GL_n(k(X))$ to the group $GL_n(\da_X)$ (through the embedding of the field $k(X)$ to the ring $K_{x,C}$
for any pair $\{ x \in C \}$).
For any irreducible curve $C$ on $X$ there is the embedding of the group $GL_n(K_C)$
to the group $GL_n(\da_X)$ (through the diagonal embedding of the field $K_C$ to the ring $K_{x,C}$ for any point $x$ on $C$ and we put $1 \in GL_n(K_{y,F})$ for any pair $\{y \in F\}$ when $F \ne C$). For any point $x$ on $X$ there is the embedding of the group $GL_n(K_x)$
to the group $GL_n(\da_X)$ (through the diagonal embedding of the ring $K_x$ to the ring $K_{x,C}$ for any irreducible curve $C \ni x$ and we put $1 \in GL_n(K_{y,F})$ for any pair $\{
y \in F \}$ when $y \ne x$).

\begin{Th}[Noncommutative reciprocity laws] \label{th1}
We consider a normal irreducible algebraic surface $X$  over a perfect field $k$.
We fix some element $a \in k(X)^*  \subset \da_X^*$ and an integer $n \ge 1$.  The central extension
 $\widehat{ GL_{n,a}(\da_X) }$
 of the group $GL_{n}(\da_X)$ by the group $k^*$
 splits canonically over the following subgroups of the group $GL_n(\da_X)$:
 \begin{itemize}
 \item[1)] \label{(1)} over the subgroup $GL_n(K_x)$ for any point $x$ on $X$,
 \item[2)] over the subgroup $GL_n(K_C)$
 for any irreducible projective curve $C$ on $X$,
 \item[3)] over the subgroup $GL_n(k(X))$ when $X$ is a projective surface.
\end{itemize}
\end{Th}
\begin{proof}
The method for the proof of this theorem is similar to the method for the proof of theorem~1
from~\cite{Os}.

We note that if $n=1$ then from the construction of the central extension  $\widehat{ GL_{1,a}(\da_X) }$
we have that it splits over the group $\da_X^*$. Therefore we can assume $n \ge 2$.

We fix a point $x \in X$. Let $\widehat{\da}_x$ be the adelic  ring  of the scheme $\Spec \hat{\oo}_x$. We have that
$$\widehat{\da}_x = \left\{ b = (b_F) \in \prod_F K_F  \; : \; b_F \in \oo_{K_F}  \quad \mbox{for all but finitely many $F$}  \right\}   $$
where $F$ runs over all prime ideals of height $1$ of the ring $\hat{\oo}_x$  and $K_F$ is the two-dimensional local field constructed by $F$.
 Clearly, $\widehat{\da}_x$ is a $2$-Tate vector space over the field $k$.
Therefore, as in~\cite[\S~4C]{OsZh} we construct a categorical central extension
$\ddet_{{{\widehat{\da}}_x}^n} \; \in \; \Ob(H^2(BGL_n(\widehat{\da}_{x}), \picz))  $.
We have that the restriction of the categorical central extension  $\ddet_{{{\widehat{\da}}_x}^n}$
to the subgroup $GL_n(K_x)$ of the group $ GL_n(\widehat{\da}_{x})$  is isomorphic to the restriction of the categorial central extension
$\ddet_{\da_{X}^n}$ to the subgroup $GL_n(K_x) $ of the group $ GL_n(\da_{X})$.  We consider a (usual) central extension
 $\widehat{ GL_{n,a}(\widehat{\da}_x) }$ of the group $GL_{n}(\widehat{\da}_x)$ by the group $k^*$ which is constructed by formulas analogous to formulas~\eqref{ff3}-\eqref{ff4} where we have to change
the $2$-Tate vector space $\da_{\Delta}^n$ to the $2$-Tate vector space ${\widehat{\da}_x}^n$. Now from the embedding  $GL_n(K_x) \subset
GL_n(\Frac{\hat{\oo}_x})$ and from above reasonings we have that the first statement of the theorem will follow from the spitting of
the central extension $\widehat{ GL_{n,a}(\widehat{\da}_x) }$ over the subgroup $GL_n(\Frac{\hat{\oo}_x}) $ of the group $ GL_{n}(\widehat{\da}_x)$.
To see the last splitting, we calculate $\{ f,g \}$ for any $f,g \in {\Frac{\hat{\oo}_x}}^*$  according to formula~\eqref{form}.
Similar to calculations in the proof of proposition~\ref{prop} we obtain $ \{ f,g \}  =C_3^{\ddet_{{{\widehat{\da}}_x}}} (f,g,a)$.  Using theorem~4.11 from~\cite{OsZh} and considering only a finite number of the height $1$ prime ideals $F$ of the ring $\hat{\oo}_x$ such that all
the poles and zeros
of elements $f$, $g$ and $a$ are among these ideals  $F$ we obtain
$$
C_3^{\ddet_{{{\widehat{\da}}_x}}} (f,g,a) = \prod_F \Nm_{k_F/k} (f,g,a)_{K_F}  \mbox{,}
$$
 where  $K_F$ is the corresponding to $F$ two-dimensional local field with the last residue field $k_F$ which is a finite extension of the field $k$.
From the reciprocity law  around a point for the two-dimensional tame symbol we have
$$\prod\limits_F \Nm_{k_F/k} (f,g,a)_{K_F}=1  \mbox{.}$$
Hence and from formula~\eqref{coh} we obtain that the central extension $\widehat{ GL_{n,a}(\widehat{\da}_x) }$ splits  over the subgroup $GL_n(\Frac{\hat{\oo}_x}) $ of the group  $ GL_{n}(\widehat{\da}_x)$. Thus, we have proved the first statement of the theorem.

We fix a projective irreducible curve $C$ on $X$.
Let $\Delta_C$ be the set of all pairs $\{x \in C\}$  when the curve $C$ is fixed. We note that
$\da_{\Delta_C} = \da_C((t_C))$,
where $\da_C$ is the adelic ring of the curve $C$ and $t_C =0$ is a local equation of the curve $C$ on some open affine subset of $X$.
From statement~\ref{delta} of proposition~\ref{restr} we have that
to prove the splitting of the central extension  $\widehat{ GL_{n,a}(\da_X) }$
over the subgroup $GL_n(K_C)$ of the group $GL_n(\da_X)$ it is enough
to prove the splitting of
the central extension $\widehat{ GL_{n,a}(\da_{\Delta_C}) }$ over the subgroup  $GL_n(K_C)$ of the group $GL_n(\da_{\Delta_C})$.
Moreover, we will assume that $a \in K_C^*$ (recall that $k(X)^*  \subset K_C^*$).

To prove the splitting of the central extension  $\widehat{ GL_{n,a}(\da_{\Delta_C}) }$
over the subgroup $GL_n(K_C)$ of the group $GL_n(\da_{\Delta_C})$ we will use formula~\eqref{form}. According to formula~\eqref{coh}, it is enough to obtain
$\{ f,g \}=1$  for any $f,g \in K_C^*$.
Similar to calculations in the proof of proposition~\ref{prop} we obtain $ \{ f,g \}  =C_3^{\ddet_{\da_{\Delta_C}}} (f,g,a)$.
From the Steinberg relation we have  $\{ h ,  h\}= \{ -1, h \}$ for any $h \in K_C^*$.
  Besides, $C_3^{\ddet_{\da_{\Delta_C}}} (\cdot, \cdot , \cdot)$  is a trimultiplicative and antisymmetric map.
  Therefore to prove
  \begin{equation}  \label{c3ca}
  C_3^{\ddet_{\da_{\Delta_C}}} (f_1,f_2,f_3) =1  \qquad  \mbox{for any} \qquad  f_1, f_2, f_3 \in K_C^*
  \end{equation}
  it is enough to consider the following two cases: 1) $f_i \in \oo_{K_C}^*$ for any $i$  and  2) $f_1, f_2 \in  \oo_{K_C}^*$, $f_3 =t_C$
  (we recall that $\oo_{K_C}$ is the discrete valuation ring of the field $K_C$). The first case follows from the fact that
  $f_i \cdot \da_C[[t_C]]= \da_C[[t_C]]$  and $\da_C[[t_C]]$ is a lattice in the $2$-Tate vector space $\da_C((t_C))$ (see lemma~4.10 from~\cite{OsZh}). Using formula~\eqref{C3} and analogs of lemma~4.12 and diagram (4-8) from~\cite{OsZh} for the $2$-Tate vector space $\da_C((t_C))$, the second case is reduced to the following
   \begin{multline*}
   C_3^{\ddet_{\da_{\Delta_C}}} (f_1,f_2,t_C)= C_3^{\ddet_{\da_{\Delta_C}}} (t_C, f_1,f_2)= \\ =
   \co(C_{t_C}^{\ddet_{\da_{\Delta_C}}})(f_1,f_2)=
   \co(\ddet_{\da_C}) (\overline{f_1}, \overline{f_2})^{-1}  \mbox{,}
   \end{multline*}
    where $\overline{f_1}, \overline{f_2}  \in k(C)^*$, i.e. it is reduced to the case of the projective curve $C$.  Now the  equality
    $\co(\ddet_{\da_C}) (\overline{f_1}, \overline{f_2}) =1 $
     follows
   from~\cite[\S~5A]{OsZh}.  Thus, we have proved the second statement of the theorem.

   The proof of the splitting of the central extension
 $\widehat{ GL_{n,a}(\da_X) }$ over the subgroup $GL_n(k(X))$ is analogous to the prove of the second statement of the theorem.
 For any $f,g \in k(X)^*$ we calculate $\{ f,g \}= C_3^{\ddet_{\da_{X}}} (f,g,a)$. Now we have
 $$C_3^{\ddet_{\da_{X}}} (f,g,a)=
 \prod_{i=1}^l C_3^{\ddet_{\da_{\Delta_{E_i}}}} (f,g,a) \mbox{,}$$
where the product is taken  over a finite set of irreducible curves $E_i$ on $X$ such that the union of the supports of divisors $(f)$, $(g)$ and $(a)$
is a subset of the union of these curves $E_i$. Now from formula~\eqref{c3ca} we obtain that $C_3^{\ddet_{\da_{\Delta_{E_i}}}} (f,g,a)$
for any $1 \le i \le l$. Thus $\{f,g \}=1$ and we have proved the third statement of the theorem. The theorem is proved.
\end{proof}

\begin{nt} {\em
From the proof of the above theorem we obtain also  the following two statements.
\begin{enumerate}
\item
 Let $x$ be a point on $X$ and $\Delta_x$
 be the set of all pairs $\{ x \in C\}$ when $x$ is fixed.
 We  fix an element  $a \in K_x^*$. Then the central extension $\widehat{ GL_{n,a}(\da_{\Delta_x}) }$ of the group $GL_n(\Delta_x)$
by the group $k^*$ splits over the subgroup $GL_n(K_x)$ of the group $GL_n(\Delta_x)$.
\item Let $C$ be an irreducible projective curve on $X$ and $\Delta_C$
 be the set of all pairs $\{ x \in C\}$ when $C$ is fixed.
 We  fix an element  $a \in K_C^*$.
 Then the central extension $\widehat{ GL_{n,a}(\da_{\Delta_C}) }$ of the group $GL_n(\Delta_C)$
by the group $k^*$ splits over the subgroup $GL_n(K_C)$ of the group $GL_n(\Delta_C)$.
\end{enumerate}
}
\end{nt}

\begin{nt} { \em
We note that the reciprocity laws with values in cohomology groups  for partial symmetrized cocycles were proved in~\cite{BM}
for a two-dimensional complex analytic space by topological methods. Therefore the reciprocity laws from theorem~\ref{th1}
are analogs  of the reciprocity laws from~\cite{BM} for an algebraic surface defined over a field of any characteristic.
}
\end{nt}

\vspace{0.3cm}

\noindent
Steklov Mathematical Institute of RAS \\
Gubkina str. 8, 119991, Moscow, Russia \\
{\it E-mail:}  ${d}_{-} osipov@mi.ras.ru$ \\

\end{document}